\documentclass{article}

\usepackage{
	amsmath, 
	amsthm, 
	amssymb,
	fancyhdr, 
	setspace,
	tikz-cd,
}
\usepackage{upgreek}

\usepackage{lscape}
\usepackage{color}
\usepackage[color,matrix,arrow]{xy}
\usepackage{setspace}
\xyoption{all}

\DeclareMathOperator{\Ext}{Ext}
\DeclareMathOperator{\Hom}{Hom}
\DeclareMathOperator{\pd}{pd}

\DeclareMathOperator{\add}{add}

\DeclareMathOperator{\End}{End}

\usepackage{avant}
\usepackage{amsmath, amssymb, amsfonts, color, tikz, bbm, txfonts, euscript}
\usepackage{amssymb}
\usepackage{enumerate}
\usepackage{bm}
\newtheorem{theorem}{Theorem}[section]
\newtheorem{lemma}[theorem]{Lemma}

\newtheorem{prop}[theorem]{Proposition}

\theoremstyle{definition}
\newtheorem{mydef}[theorem]{Definition}

\begin{document}

\thispagestyle{empty}

\title{$\uptau$-Tilting Finite Cluster-Tilted Algebras}
\author{Stephen Zito}
        
\maketitle

\begin{abstract}
We prove if $B$ is a cluster-tilted algebra, then $B$ is $\tau_B$-tilting finite if and only if $B$ is representation-finite. 
\end{abstract}

\section{Introduction} 
The theory of $\tau$-tilting was introduced by Adachi, Iyama and Reiten in $\cite{AIR}$ as a far-reaching generalization of classical tilting theory for finite dimensional associative algebras.  One of the main classes of objects in the theory is that of $\tau$-rigid modules: a module $M$ over an algebra $\Lambda$ is $\it{\tau_{\Lambda}}$-$\it{rigid}$ if $\mathop{\text{Hom}_{\Lambda}(M,\tau_{\Lambda}M)}=0$, where $\tau_{\Lambda}M$ denotes the Auslander-Reiten translation of $M$; such a module $M$ is called $\it{\tau_{\Lambda}}$-$\it{tilting}$ if the number $|M|$ of non-isomorphic indecomposable summands of $M$ equals the number of isomorphism classes of simple $\Lambda$-modules.  Recently, a new class of algebras were introduced by Demonet, Iyama, Jasso in $\cite{DIJ}$ called $\it{\tau_{\Lambda}}$-$\it{tilting~finite}$ algebras.  They are defined as finite dimensional algebras with only a finite number of isomorphism classes of basic $\tau_{\Lambda}$-tilting modules.
 \par
An obvious sufficient condition for an algebra to be $\tau_{\Lambda}$-titling finite is for it to be representation-finite.  In general, this condition is not necessary.  The aim of this note is to prove for $\it{cluster}$-$\it{tilted~algebras}$, this condition is in fact necessary. 
\par
Tilted algebras are the endomorphism algebras of tilting modules over hereditary algebras, introduced by Happel and Ringel in $\cite{HR}$.  Cluster-tilted algebras are the endomorphism algebras of cluster-tilting objects over cluster categories of hereditary algebras, introduced by Buan, Marsh and Reiten in $\cite{BMR}$. The similarity in the two definitions lead to the following precise relation between tilted and cluster-tilted algebras, which was established in $\cite{ABS}$ by Assem, Br$\ddot{\text{u}}$stle, and Schiffler.
\par
There is a surjective map
\[
\{\text{tilted}~\text{algebras}\}\longmapsto \{\text{cluster-tilted}~\text{algebras}\}
\]
\[
C\longmapsto B=C\ltimes E
\]
where $E$ denotes the $C$-$C$-bimodule $E=\text{Ext}_C^2(DC,C)$ and $C\ltimes E$ is the trivial extension.
\par
This result allows one to define cluster-tilted algebras without using the cluster category.  Using this construction, we show the following.  
\begin{theorem}
\label{Main1}
 Let $B$ be a cluster-tilted algebra.  Then $B$ is $\tau_B$-tilting finite if and only if $B$ is representation-finite.  
 \end{theorem}

\section{Notation and Preliminaries}
We now set the notation for the remainder of this paper. All algebras are assumed to be finite dimensional over an algebraically closed field $k$.  If $\Lambda$ is a $k$-algebra then denote by $\mathop{\text{mod}}\Lambda$ the category of finitely generated right $\Lambda$-modules and by $\mathop{\text{ind}}\Lambda$ a set of representatives of each isomorphism class of indecomposable right $\Lambda$-modules.  We denote by $\add M$ the smallest additive full subcategory of $\mathop{\text{mod}}\Lambda$ containing $M$, that is, the full subcategory of $\mathop{\text{mod}}\Lambda$ whose objects are the direct sums of direct summands of the module $M$.  Given $M\in\mathop{\text{mod}}\Lambda$, the projective dimension of $M$ is denoted $\pd_{\Lambda}M$.  We let $\tau_{\Lambda}$ and $\tau^{-1}_{\Lambda}$ be the Auslander-Reiten translations in $\mathop{\text{mod}}\Lambda$.  We let $D$ be the standard duality functor $\Hom_k(-,k)$.  Finally, $\Gamma(\mathop{\text{mod}}\Lambda)$ will denote the Auslander-Reiten quiver of $\Lambda$.
 
 \subsection{Tilted Algebras}
  Tilting theory is one of the main themes in the study of the representation theory of algebras.  Given a $k$-algebra $A$, one can construct a new algebra $B$ in such a way that the corresponding module categories are closely related.  The main idea is that of a tilting module.
   \begin{mydef} Let $A$ be an algebra.  An $A$-module $T$ is a $\emph{partial tilting module}$ if the following two conditions are satisfied: 
   \begin{enumerate}
   \item[($\text{1}$)] $\pd_AT\leq1$.
   \item[($\text{2}$)] $\Ext_A^1(T,T)=0$.
   \end{enumerate}
   A partial tilting module $T$ is called a $\emph{tilting module}$ if it also satisfies the following additional condition:
   \begin{enumerate}
   \item[($\text{3}$)] There exists a short exact sequence $0\rightarrow A\rightarrow T'\rightarrow T''\rightarrow 0$ in $\mathop{\text{mod}}A$ with $T'$ and $T''$ $\in \add T$.
   \end{enumerate}
   \end{mydef}

We now state the definition of a tilted algebra.
 \begin{mydef} Let $A$ be a hereditary algebra with $T$ a tilting $A$-module.  Then the algebra $B=\End_AT$ is called a $\emph{tilted algebra}$.
 \end{mydef}

 \subsection{Cluster categories and cluster-tilted algebras} Let $C=kQ$ be the path algebra of the quiver $Q$ and let $\mathcal{D}^b(\mathop{\text{mod}}C)$ denote the derived category of bounded complexes of $C$-modules.  The $cluster$ $category$ $\mathcal{C}_C$ is defined as the orbit category of the derived category with respect to the functor $\tau_\mathcal{D}^{-1}[1]$, where $\tau_\mathcal{D}$ is the Auslander-Reiten translation in the derived category and $[1]$ is the shift.  Cluster categories were introduced in $\cite{BMMRRT}$, and in $\cite{CCS}$ for type $\mathbb{A}$.
 \par
 An object $T$ in $\mathcal{C}_C$ is called $cluster$-$tilting$ if $\text{Ext}_{\mathcal{C}_C}^1(T,T)=0$ and $T$ has $|Q_0|$ non-isomorphic indecomposable direct summands where $|Q_0|$ is the number of vertices of $Q$.  The endomorphism algebra $\End_{\mathcal{C}_C}T$ of a cluster-tilting object is called a $cluster$-$tilted$ $algebra$ $\cite{BMR}$.

\subsection{Relation extensions}
 Let $C$ be an algebra of global dimension at most 2 and let $E$ be the $C$-$C$-bimodule $E=\text{Ext}_C^2(DC,C)$.  
 \begin{mydef} The $\emph{relation extension}$ of $C$ is the trivial extension $B=C\ltimes E$, whose underlying $C$-module structure is $C\oplus E$, and multiplication is given by $(c,e)(c^{\prime},e^{\prime})=(cc^{\prime},ce^{\prime}+ec^{\prime})$.
 \end{mydef}  
 
 Relation extensions were introduced in $\cite{ABS}$.  In the special case where $C$ is a tilted algebra, we have the following result.
 \begin{theorem}$\emph{\cite[Theorem~3.4]{ABS}}$.  Let C be a tilted algebra.  Then $B=C\ltimes\emph{Ext}_C^2(DC,C)$ is a cluster-tilted algebra.  Moreover all cluster-tilted algebras are of this form.
 \end{theorem}

\subsection{Slices and local slices} 
 \begin{mydef}
  A $\it{slice}$ $\Sigma$ in $\Gamma(\mathop{\text{mod}}A)$ is a set of indecomposable $A$-modules such that
 \begin{enumerate}
\item[($\text{1}$)] $\Sigma$ is sincere.
\item[($\text{2}$)] Any path in $\mathop{\text{mod}}A$ with source and target in $\Sigma$ consists entirely of modules in $\Sigma$.
\item[($\text{3}$)] If $M$ is an indecomposable non-projective $A$-module then at most one of $M$ , $\tau_AM$ belongs to $A$.
\item[($\text{4}$)] If $M\rightarrow S$ is an irreducible morphism with $M,S\in\mathop{\text{ind}}A$ and $S\in\Sigma$, then either $M$ belongs to $\Sigma$ or $M$ is non-injective and $\tau_A^{-1}M$ belongs to $\Sigma$.
\end{enumerate}  
\end{mydef}  
The existence of slices is used to characterize tilted algebras in the following way.
 \begin{theorem}$\emph{\cite{R}}$
 \label{tiltslice}
 Let $B=\emph{End}_AT$ be a tilted algebra.  Then the class of $B$-modules $\emph{Hom}_A(T,DA)$ forms a slice in $\mathop{\text{mod}}B$.  Conversely, any slice in any module category is obtained in this way.
 \end{theorem}
 The following notion of local slices was introduced in $\cite{ABS2}$ in the context of cluster-tilted algebras.  We say a path $X=X_0\rightarrow X_1\rightarrow X_2\rightarrow\dots\rightarrow X_s=Y$ in $\Gamma(\mathop{\text{mod}}A)$ is $\it{sectional}$ if, for each $i$ with $0<i<s$, we have $\tau_AX_{i+1}\neq X_{i-1}$. 
  \begin{mydef}
 A $\it{local~slice}$ $\Sigma$ in $\Gamma(\mathop{\text{mod}}A)$ is a set of indecomposable $A$-modules inducing a connected full subquiver of $\Gamma(\mathop{\text{mod}}A)$ such that
 \begin{enumerate}
\item[($\text{1}$)] If $X\in\Sigma$ and $X\rightarrow Y$ is an arrow in $\Gamma(\mathop{\text{mod}}A)$, then either $Y$ or $\tau_AY\in\Sigma$.
\item[($\text{2}$)] If $Y\in\Sigma$ and $X\rightarrow Y$ is an arrow in $\Gamma(\mathop{\text{mod}}A)$, then either $X$ or $\tau_A^{-1}X\in\Sigma$.
\item[($\text{3}$)] For every sectional path $X=X_0\rightarrow X_1\rightarrow X_2\rightarrow\dots\rightarrow X_s=Y$ in $\Gamma(\mathop{\text{mod}}A)$ with $X,Y\in\Sigma$, we have $X_i\in\Sigma$, for $i=0,1,\dots,s.$
\item[($\text{4}$)] The number of indecomposable $A$-modules in $\Sigma$ equals the number of non-isomorphic summands of $T$, where $T$ is a tilting $A$-module.
\end{enumerate}  
 \end{mydef}
 There is a relationship between tilted and cluster-tilted algebras given in terms of slices and local slices.
 \begin{theorem}$\emph{\cite[Corollary~20]{ABS2}}$
 \label{clusterlocal}
  Let $C$ be a tilted algebra and $B$ the corresponding cluster-tilted algebra.  Then any slice in $\mathop{\emph{mod}}C$ embeds as a local slice in $\mathop{\emph{mod}}B$ and any local slice $\Sigma$ in $\mathop{\emph{mod}}B$ arises in this way.
 \end{theorem}
  The existence of local slices in a cluster-tilted algebra gives rise to the following definition.  The unique connected component of $\Gamma(\mathop{\text{mod}}B)$ that contains local slices is called the $\it{transjective~component}$.  
  
The next result says a slice in a tilted algebra together with its $\tau$ and $\tau^{-1}$ translates full embeds in the cluster-tilted algebra.

\begin{prop}$\emph{\cite[Proposition~3]{ABS3}}$
\label{Same}
Let $C$ be a tilted algebra, $\Sigma$ a slice, $M\in\Sigma$, and $B$ the corresponding cluster-tilted algebra.
\begin{enumerate} 
\item[($\text{1}$)] $\tau_CM\cong\tau_BM$.
\item[($\text{2}$)] $\tau_C^{-1}M\cong\tau_B^{-1}M$.
\end{enumerate}
\end{prop}

In $\cite{ABS2}$, the authors gave an example of an indecomposable transjective module over a cluster-tilted algebra that does not lie on a local slice.  It was proved in $\cite{ASS1}$ the number of such modules is finite.
\begin{prop}$\emph{\cite[Corollary~3.8]{ASS1}}$
\label{local}
Let $B$ be a cluster-tilted algebra.  Then the number of isomorphism classes of indecomposable transjective $B$-modules that do not lie on a local slice is finite.
\end{prop}

 \subsection{$\tau$-tilting finite algebras}
Following $\cite{AIR}$ we state the following definition.
\begin{mydef}  A $C$-module $M$ is $\tau_C$-$\emph{rigid}$ if $\text{Hom}_C(M,\tau_CM)=0$.  A $\tau_C$-rigid module $M$ is $\tau_C$-$\emph{tilting}$ if the number of pairwise, non-isomorphic, indecomposable summands of $M$ equals the number of isomorphism classes of simple $C$-modules.
 \end{mydef}  
It follows from the Auslander-Reiten formulas that any $\tau_C$-rigid module is rigid and the converse holds if the projective dimension is at most 1.  In particular, any partial tilting module is a $\tau_C$-rigid module, and any tilting module is a $\tau_C$-tilting module.  Thus, we can regard $\tau_C$-tilting theory as a generalization of  classic tilting theory. 
Following $\cite{DIJ}$, we have the following definition.
\begin{mydef}
Let $A$ be a finite dimensional algebra. We say that $A$ is $\tau_A$-$\it{tilting~finite}$ if there are only finitely many isomorphism classes of basic $\tau_A$-tilting $A$-modules.
\end{mydef}
The authors provide several equivalent conditions for an algebra $A$ to be $\tau_A$-tilting finite.  In particular, we need the following.
\begin{lemma}$\emph{\cite[Corollary~2.9.]{DIJ}}$
\label{tau-tilt}
$A$ is $\tau_A$-tilting finite if and only if there are only finitely many isomorphism classes of indecomposable $\tau_A$-rigid $A$-modules.
\end{lemma}
\subsection{A criterion for representation-finiteness}
We will need the following criterion for an algebra to be representation-finite.
\begin{theorem}$\emph{\cite[IV~Theorem~5.4.]{ASS}}$
\label{finite}
Assume $A$ is a basic and connected finite dimensional algebra.  If $\Gamma(\mathop{\emph{mod}}A)$ admits a  finite connected component $\mathcal{C}$, then $\mathcal{C}=\Gamma(\mathop{\emph{mod}}A)$.  In particular, $A$ is representation-finite.
\end{theorem}

\section{Main Result}
We are now ready to prove our main theorem.
\begin{theorem}
Let $B$ be a cluster-tilted algebra.  Then $B$ is $\tau_B$-tilting finite if and only if $B$ is representation-finite.
\end{theorem}
\begin{proof}
The sufficiency is obvious so we prove the necessity.  Assume $B$ is $\tau_B$-tilting finite but representation-infinite.  By Theorems $\ref{tiltslice}$ and $\ref{clusterlocal}$, we know the transjective component of $\Gamma(\mathop{\text{mod}}B)$ exists.  Since $B$ is representation-infinite, Theorem $\ref{finite}$ guarantees the transjective component must be infinite.  By Proposition $\ref{local}$ and the fact that the transjective component is infinite, we must have an infinite number of indecomposable transjective $B$-modules which lie on a local slice. Let $M$ be such a $B$-module.  Theorem $\ref{clusterlocal}$ guarantees there exists a tilted algebra $C$ and a slice $\Sigma$ such that $M$ is a $C$-module and $M\in\Sigma$.  It follows from parts $(2)$ and $(3)$ of the definition of a slice that $M$ is $\tau_C$-rigid.  By Proposition $\ref{Same}$, we know $\tau_CM\cong\tau_BM$.  This implies $M$ is $\tau_B$-rigid.  Since $M$ was arbitrary, we have shown there exists an infinite number of indecomposable transjective $B$-modules which are $\tau_B$-rigid.  This is a contradiction to our assumption that $B$ was $\tau_B$-tilting finite and Lemma $\ref{tau-tilt}$.  We conclude $B$ must be representation-finite.

\end{proof}

\noindent Department of Mathematics, University of Connecticut-Waterbury, Waterbury, CT 06702, USA
\\
\it{E-mail address}: \bf{stephen.zito@uconn.edu}

\end{document}